\documentclass[reqno]{amsart}

\usepackage{mystyle}

\usepackage{etoolbox}

\usepackage[color=cyan!50]{todonotes} 
\presetkeys{todonotes}{inline}{} 
\usepackage{xcolor}

\usepackage[headings]{fullpage}
\usepackage{parskip}

\usepackage[utf8]{inputenc}
\usepackage[T1]{fontenc}
\usepackage[british]{babel}
\usepackage[pdfencoding=auto]{hyperref}

\usepackage{microtype}
\usepackage{booktabs}

\usepackage{amsfonts,amssymb,bbm,mathrsfs,stmaryrd}
\usepackage{amsmath}

\usepackage{mathtools}
\usepackage[noabbrev]{cleveref}
\usepackage{centernot}
\usepackage[shortlabels]{enumitem}
\usepackage{url}
\usepackage{tikz}
\usepackage{pict2e}
\usepackage{framed}

\usetikzlibrary{matrix}
\usetikzlibrary{positioning}
\usetikzlibrary{arrows}
\tikzset{> =stealth}
\usetikzlibrary{arrows.meta}
\tikzset{normalHead/.tip={Triangle[open,angle=60:4pt]},}
\tikzset{normalTail/.tip={Triangle[reversed,open,angle=60:4pt]},}

\theoremstyle{plain}
\newtheorem{theorem}{Theorem}[section]

\newtheorem{proposition}[theorem]{Proposition}
\newtheorem{corollary}[theorem]{Corollary}

\theoremstyle{definition}
\newtheorem{definition}[theorem]{Definition}

\newtheorem{example}{Example}
\AtBeginEnvironment{example}{\vspace{11pt}\begin{leftbar}\vspace{-11pt}}
\AtEndEnvironment{example}{\end{leftbar}}

\AtBeginEnvironment{construction}{\begin{leftbar}}
\AtEndEnvironment{construction}{\end{leftbar}}

\theoremstyle{remark}

\renewcommand{\epsilon}{\varepsilon}
\renewcommand{\phi}{\varphi}

\newcommand{\inv}{^{-1}}

\newcommand{\WSExt}{\mathrm{WSExt}}

\newcommand{\SpltExt}{\mathrm{SplExt}}

\newcommand{\ophi}{\overline{\phi}}

\DeclareMathOperator{\Gl}{Gl}

\newcommand{\splitext}[6]{
\tikz[baseline]{
\newdimen{\mylabelwidth}
\newdimen{\skipwidth}
\node[anchor=base] (A) {\hspace*{\dimexpr0.5pt-\pgfkeysvalueof{/pgf/inner xsep}}${#1}$};
\settowidth{\mylabelwidth}{\pgfinterruptpicture {$#2$} \endpgfinterruptpicture}
\pgfmathsetlength{\skipwidth}{max(\mylabelwidth,10pt)}
\node[right] (B) at ([xshift=\skipwidth+12pt]A.east) {${#3}$};
\settowidth{\mylabelwidth}{\pgfinterruptpicture {$#4$} \endpgfinterruptpicture}
\settowidth{\skipwidth}{\pgfinterruptpicture {$#5$} \endpgfinterruptpicture}
\pgfmathsetlength{\skipwidth}{max(\skipwidth,\mylabelwidth,10pt)}
\node[right] (C) at ([xshift=\skipwidth+12pt]B.east) {${#6}$\hspace*{\dimexpr0.5pt-\pgfkeysvalueof{/pgf/inner xsep}}};
\draw[normalTail->] (A) to node [above] {${#2}$} (B);
\draw[transform canvas={yshift=0.5ex},-normalHead] (B) to node [above] {${#4}$} (C);
\draw[transform canvas={yshift=-0.5ex},->] (C) to node [below] {${#5}$} (B);
}}

\title{A characterization of weakly Schreier extensions of monoids}
\author[P. F. Faul]{Peter F. Faul}
\address{Department of Pure Mathematics and Statistical Sciences\\ University of Cambridge}
\email{peter@faul.io}

\date{\today}


\begin{document}

\maketitle

\begin{abstract}
A split extension of monoids with kernel $k \colon N \to G$, cokernel $e \colon G \to H$ and splitting $s \colon H \to G$ is Schreier if there exists a unique set-theoretic map $q \colon G \to N$ such that for all $g \in G$, $g = kq(g) \cdot se(g)$. Schreier extensions have a complete characterization and have been shown to correspond to monoid actions of $H$ on $N$. If the uniqueness requirement of $q$ is relaxed, the resulting split extension is called weakly Schreier. A natural example of these is the Artin glueings of frames. In this paper we provide a complete characterization of the weakly Schreier extensions of $H$ by $N$, proving them to be equivalent to certain quotients of $N \times H$ paired with a function that behaves like an action with respect to the quotient. Furthermore, we demonstrate the failure of the split short lemma in this setting and provide a full characterization of the morphisms that occur between weakly Schreier extensions. Finally, we use the characterization to construct some classes of examples of weakly Schreier extensions.
\end{abstract}

\section{Introduction}


It is well understood that for groups $H$ and $N$, the semidirect product construction provides an equivalence between actions of $H$ on $N$ and split extensions of $H$ by $N$. The same cannot be said when $H$ and $N$ are replaced with monoids; however, monoid actions do correspond naturally to a certain class of split extensions of monoids: the Schreier split extensions. These split extensions of monoids were first alluded to in \cite{patchkoria1998crossed} and were first studied explicitly in \cite{martins2013semidirect}, where their relationship to actions was established. We briefly sketch one direction of this relationship below.

A \emph{Schreier extension} $\splitext{N}{k}{G}{e}{s}{H}$ is a split extension in which for all $g \in G$, there exist unique $n \in N$ such that $g = k(n) \cdot se(g)$. When $n$ is not required to be unique, we call the resulting split extension \emph{weakly Schreier}.

Given a Schreier extension $\splitext{N}{k}{G}{e}{s}{H}$, we can associate to it a set-theoretic map $q$ which satisfies that for all $g \in G$, $g = kq(g) \cdot se(g)$. From $q$ we can construct an action $\alpha \colon H \times N \to N$ where $\alpha(h,n) = q(s(h)k(n))$, which can then be used to define a multiplication on the set $N \times H$ given by $(n,h)\cdot (n',h') = (n \cdot \alpha(h,n'),hh')$. The result is a monoid $(N \times H, \cdot, (1,1))$ isomorphic to $G$ via $\phi \colon (N \times H, \cdot, (1,1)) \to G$, where $(n,h)$ is sent to $k(n) \cdot s(h)$.

\subsection*{\texorpdfstring{$S$}{S}-protomodularity}

\emph{Pointed protomodular categories} (see \cite{bourn1998protomodularity} and \cite{borceux2004protomodular}) may be thought of as categories with will behaved split extensions. The study of Schreier split extensions motivated a more relaxed notion, that of pointed $S$-protomodularity, where only a restricted class $S$ of split extensions need be well behaved \cite{bourn2015Sprotomodular}. 

Just what properties this class $S$ must satisfy has not been firmly established. When inspiration is taken from Schreier split epimorphisms we require that, in addition to other properties, $S$ be closed under the taking of finite limits. However in \cite{bourn2015partialMaltsev}, a situation is considered in which this property is relaxed. The latter situation captures the case of weakly Schreier extensions whereas the former does not. For this reason one might wonder whether there is good motivation for the study of weakly Schreier extensions. In the following paragraphs we justify their study.

\subsection*{Artin glueings as weakly Schreier extensions}

The \emph{Artin glueing} of two frames $N$ and $H$ along a finite meet preserving map $f \colon H \to N$, is a well known construction which returns the frame of pairs $(n,h)$ satisfying that $n \leq f(h)$ with componentwise meet and join \cite{wraith1974glueing}.

In \cite{faul2019artin} it was shown that in the category of frames with finite meet preserving maps, Artin glueings are equivalent to a certain class of split epimorphisms. A split epimorphism $\splitext{N}{k}{G}{e}{s}{H}$ belongs to this class if and only if 
for each $g \in G$, there exists an $n \in N$ such that $g = k(n)se(g)$.

Note that the category we are considering is a full subcategory of the category of monoids where we consider frames to be monoids with operation given by the finite meet. Then Artin glueings of $H$ by $N$ precisely correspond to the weakly Schreier extensions of $H$ by $N$ which occur in this subcategory.

Along with Schreier extensions, Artin glueings provide a natural example of weakly Schreier extensions and motivate their study.

\subsection*{Outline}

In this paper we provide a complete classification of the weakly Schreier extensions of $H$ by $N$ proving them equivalent to certain quotients of $N \times H$, equipped with something that behaves like an action relative to the quotient. Further, we demonstrate the failure of the split short five lemma and provide a complete classification of the morphisms that occur between two weakly Schreier extensions. Finally we provide some techniques for constructing weakly Schreier extensions. First by generalising the Artin glueing construction and then by considering the coarsest quotient compatible with our construction. 

Hopefully, with the better understanding of weakly Schreier extensions that this paper provides, progress might be made towards resolving certain issues surrounding $S$-protomodularity.

\section{Weak semidirect products}\label{weaksemidirectproducts}

Inspired by the semidirect product construction for Schreier extensions, we consider a related construction in the weakly Schreier setting. 

\begin{definition}
A diagram $\splitext{N}{k}{G}{e}{s}{H}$ is a \emph{split extension} if 
\begin{enumerate}
    \item $k$ is the kernel of $e$,
    \item $e$ is the cokernel of $k$,
    \item $es = 1_H$.
\end{enumerate}
\end{definition}

\begin{definition}
The category $\SpltExt(H,N)$ has split extensions of $H$ by $N$ as objects and, as morphisms, monoid maps $f \colon G_1 \to G_2$ making the three squares in the following diagram commute.

\begin{center}
   \begin{tikzpicture}[node distance=1.5cm, auto]
    \node (A) {$N$};
    \node (B) [right of=A] {$G_1$};
    \node (C) [right of=B] {$H$};
    \node (D) [below of=A] {$N$};
    \node (E) [right of=D] {$G_2$};
    \node (F) [right of=E] {$H$};
    \draw[normalTail->] (A) to node {$k_1$} (B);
    \draw[transform canvas={yshift=0.5ex},-normalHead] (B) to node {$e_1$} (C);
    \draw[transform canvas={yshift=-0.5ex},->] (C) to node {$s_1$} (B);
    \draw[normalTail->] (D) to node {$k_2$} (E);
    \draw[transform canvas={yshift=0.5ex},-normalHead] (E) to node {$e_2$} (F);
    \draw[transform canvas={yshift=-0.5ex},->] (F) to node {$s_2$} (E);
    \draw[->] (B) to node {$f$} (E);
    \draw[double equal sign distance] (A) to (D);
    \draw[double equal sign distance] (C) to (F);
   \end{tikzpicture}
\end{center}

\end{definition}

\begin{definition}
A split extension $\splitext{N}{k}{G}{e}{s}{H}$ is called weakly Schreier when every element of $g \in G$ can be written as $g = k(n) \cdot se(g)$ for some $n \in N$.
\end{definition}

Recall that if $n$ is required to be unique, then the extension is called Schreier.

Similar to Schreier extensions, at least under the assumption of the axiom of choice, the definition can be reframed in terms of a set-theoretic map $q$. However, in the weakly Schreier setting, this map will not in general be unique. 

\begin{proposition}\label{prop:axiom}
Under the assumption of the axiom of choice, an extension $\splitext{N}{k}{G}{e}{s}{H}$ is weakly Schreier if and only if there exists a set theoretic map $q \colon G \to N$ satisfying that for all $g \in G$, $g = kq(g) \cdot se(g)$.
\end{proposition}

Inspired by \cite{bourn2015Sprotomodular} we call such a map $q$, an \emph{associated Schreier retraction} of $\splitext{N}{k}{G}{e}{s}{H}$. We now prove some analogues of results that occur in \cite{bourn2015Sprotomodular}.

\begin{proposition}\label{prop:reviewersuggestion}
Let $\splitext{N}{k}{G}{e}{s}{H}$ be weakly Schreier and let $q$ be an associated Schreier retraction. Then the following properties hold:
\begin{enumerate}
    \item $qk = 1_N$,
    \item $q(1) = 1$,
    \item $kq(s(h)k(n)) \cdot s(h) = s(h)k(n).$
\end{enumerate}
\end{proposition}

\begin{proof}
(1) Per the definition of $q$, for each $n \in N$ we can write $k(n) = kqk(n) \cdot sek(n) = kqk(n)$. Since $k$ is injective we find that $n = qk(n)$, and so $q$ is a retraction of $k$.

(2) We know that $1 = q(1) \cdot se(1) = q(1)$.

(3) Notice that $e(s(h)k(n)) = h$. Thus we can write $s(h)k(n) = kq(s(h)k(n) \cdot se(s(h)k(n)) = kq(s(h)k(n)) \cdot s(h)$.
\end{proof}

We will make extensive use of \cref{prop:axiom} in \cref{characterization}. For the remainder of this section, as well as in \cref{splitshortfive} we will work choice free. It is likely that, with some thought, the results in \cref{characterization} can be presented in a choice free manner too.

\begin{definition}
The category $\WSExt(H,N)$ is the full subcategory of $\SpltExt(H,N)$ consisting of the weakly Schreier extensions.
\end{definition}

\subsection{Canonical quotient}

Let $\splitext{N}{k}{G}{e}{s}{H}$ be a weakly Schreier extension and consider the set-function $\phi \colon N \times H \to G$ sending $(n,h)$ to $k(n) \cdot s(h)$. The weakly Schreier condition gives that $\phi$ is surjective and so we can quotient $N \times H$ by $\phi$.

\begin{definition}
Let $\splitext{N}{k}{G}{e}{s}{H}$ be a weakly Schreier extension and let $\phi \colon N \times H \to G$ denote the surjective map sending $(n,h)$ to $k(n) \cdot s(h)$. Then let $E(e,s)$ denote the equivalence relation on $N \times H$ induced by $\phi$ where $(n,h) \sim (n',h')$ if and only if $k(n) \cdot s(h) = k(n') \cdot s(h')$. 	
\end{definition}

Thus we can consider the map $\ell \colon N \times H \to (N \times H)/E(e,s)$, where $\ell$ send $(n,h)$ to $[n,h]$, the equivalence class of $(n,h)$ with respect to $E(e,s)$. Naturally we have a bijection $\ophi \colon (N \times H)/E(e,s) \to G$ such that $\ophi\ell = \phi$. 

Before we equip $(N \times H)/E(e,s)$ with a multiplication let us study some properties of $E(e,s)$.

\begin{proposition}\label{sameh}
Let $\splitext{N}{k}{G}{e}{s}{H}$ be a weakly Schreier extension. If $(n_1,h_1) \sim (n_2,h_2)$ in $E(e,s)$, then $h_1=h_2$.
\end{proposition}

\begin{proof}
Let $(n_1,h_1) \sim (n_2,h_2)$. Then we have that $k(n_1) \cdot s(h_1) = k(n_2) \cdot s(h_2)$. Applying $e$ to both sides yields $h_1 = h_2$ as required.
\end{proof}

\begin{proposition}\label{samen}
Let $\splitext{N}{k}{G}{e}{s}{H}$ be a weakly Schreier extension. If $(n_1,1) \sim (n_2,1)$ in $E(e,s)$, then $n_1 = n_2$.
\end{proposition}

\begin{proof}
Let $(n_1,1) \sim (n_2,1)$. Then we have that $k(n_1) \cdot s(1) = k(n_2) \cdot s(1)$. Since $s(1) = 1$ this gives that $k(n_1) = k(n_2)$ which further implies that $n_1 = n_2$, as $k$ is injective.
\end{proof}

Combining \cref{sameh} and \cref{samen} we get that for each $n \in N$, the equivalence class $[n,1]$ is a singleton.

\begin{proposition}\label{welldefinedaction}
Let $\splitext{N}{k}{G}{e}{s}{H}$ be a weakly Schreier extension. If $(n_1,h) \sim (n_2,h)$ in $E(e,s)$ then for all $n \in N$ we have $(nn_1,h) \sim (nn_2,h)$, and for all $h' \in H$ we have $(n_1,hh') \sim (n_2,hh')$.	
\end{proposition}

\begin{proof}
Let $(n_1,h) \sim (n_2,h)$. Then we have that $k(n_1) \cdot s(h) = k(n_2) \cdot s(h)$. Hence $k(n) \cdot k(n_1) \cdot s(h) = k(n) \cdot k(n_2) \cdot s(h)$ and $k(n_1) \cdot s(h) \cdot s(h') = k(n_2) \cdot s(h) \cdot s(h')$ which gives that $(nn_1,h) \sim (nn_2,h)$ and that $(n_1,hh') \sim (n_2,hh')$.	
\end{proof}

Now let us discuss the monoid structure of $N \times H /E(e,s)$. It inherits its multiplication and identity from $G$ through $\ophi$ --- that is, we define $[n,h] \cdot [n',h'] = \ophi\inv(\ophi([n,h])\ophi([n',h'])).$ The identity is readily seen to be $[1,1]$. By construction $\ophi$ preserves multiplication and so we see that $((N \times H)/E(e,s),\cdot,[1,1])$ is isomorphic to $G$. In \cref{characterization} we will give an explicit description of this multiplication, making use of \cref{prop:axiom}.

For convenience, we now let $(N \times H)/E(e,s)$ denote the monoid $((N \times H)/E(e,s),\cdot,[1,1])$ and we might think of it as a weaker notion of a semidirect product. We call this construction a \emph{weak semidirect product} of $H$ by $N$. This choice of terminology will be fully justified in \cref{characterization} after the full characterization of weakly Schreier extensions.

In fact $(N \times H)/E(e,s) $ is part of a split extension isomorphic to $\splitext{N}{k}{G}{e}{s}{H}$. Consider \[\splitext{N}{k'}{(N \times H)/E(e,s)}{e'}{s'}{H},\] where $k'(n) = [n,1]$, $e'([n,h]) = h$ and $s'(h) = [1,h]$. By \cref{sameh} we have that $e'$ is well defined. Note that $[n,1] \cdot [1,h] = [n,h]$ and so this extension is weakly Schreier by construction.

We can now conclude the following.

\begin{proposition}\label{ophiiso}
The map $\ophi$ is an isomorphism of split extensions between $\splitext{N}{k}{G}{e}{s}{H}$ and \\$\splitext{N}{k'}{(N \times H)/E(e,s)}{e'}{s'}{H}$.	
\end{proposition}

%
%

\begin{example}{(Schreier extensions)}\label{Schreierexample}
Let $\splitext{N}{k}{G}{e}{s}{H}$ be a Schreier extension and $q \colon G \to N$ the associated Schreier retraction. Recall that the semidirect product construction (as in \cite{martins2013semidirect}) applied to the above Schreier extension, gives $(N \times H, \cdot, (1,1))$ where $(n,h) \cdot (n',h') = (n \cdot q(s(h)k(n')),hh')$. Let us show that this agrees with our weak semidirect product construction on weakly Schreier extensions.

Let $\splitext{N}{k}{G}{e}{s}{H}$ be a Schreier extension, let $\phi$ be defined as above and consider $(N \times H)/E(e,s)$. Observe that the Schreier condition gives that $(n,h) \sim (n',h')$ if and only if $(n,h) = (n',h')$. Thus $N \times H/E(e,s)$ as a set is just the product $N \times H$. This agrees with the semidirect product construction and so we must just show that the two constructions agree on multiplication.

For our weak semidirect product we define $(n,h) \cdot (n',h') = \phi\inv(\phi(n,h)\phi(n',h'))$ which is the unique element in $N \times H$ which $\phi$ sends to $k(n) \cdot s(h) \cdot k(n') \cdot s(h')$. Thus we need only show that $\phi(n \cdot q(s(h)k(n')),hh') = k(n) \cdot s(h) \cdot k(n') \cdot s(h')$.

Per the definition of $q$ we see that $kq(s(h)k(n)) \cdot s(h) = s(h) \cdot k(n)$. Thus 
\begin{align*}
	\phi(n \cdot q(s(h)k(n')),hh') &= k(n) \cdot kq(s(h)k(n')) \cdot s(h) \cdot s(h') \\
	&= k(n) \cdot s(h) \cdot k(n') \cdot s(h').
\end{align*}

Thus we see that the weak semidirect product construction agrees with the semidirect product construction on Schreier extensions.
\end{example}

\begin{example} {(Artin glueings)}\label{Artin}

Artin glueings are our primary examples of weakly Schreier extensions that are not Schreier. Artin glueings are usually considered as subobjects of the product and not quotients. For frames $N$ and $H$ and finite meet preserving map $f \colon H \to N$, the Artin glueing $\Gl(f)$ is the frame of pairs $(n,h)$ where $n \le f(h)$, with componentwise meets and joins.
In this example we discuss how an Artin glueing can also be viewed as a quotient of the product.

Let $N,G$ and $H$ be frames considered as monoids with multiplication given by meet and consider a weakly Schreier extension $\splitext{N}{k}{G}{e}{s}{H}$. As shown in \cite{faul2019artin}, $k$ has a left adjoint $k^*$ which is an associated Schreier retraction and, in fact, a monoid map. Furthermore we find that $s$ must be the right adjoint $e_*$ of $e$.

Now let $\phi \colon N \times H \to G$ be defined in the usual way and consider the equivalence classes it generates. We know that $kk^*(g)\wedge e_*e(g) = g$ and thus the equivalence class corresponding to each $g$ has a canonical choice of representative given by $(k^*(g),e(g))$. This allows us to represent the inverse $\ophi\inv(g) = [k^*(g),e(g)]$.

Starting with the fact that for all $g$ we have that $g \le e_*e(g)$, we apply $k^*$ to both sides and arrive at $k^*(g) \le k^*e_*e(g)$. This means that all of the canonical elements $(k^*(g),e(g))$ are elements of $\Gl(k^*e_*)$. Furthermore, if $(n,h) \in \Gl(k^*e_*)$ then we have $k^*(k(n) \wedge e_*(h)) = k^*k(n) \wedge k^*e_*(h) = n \wedge k^*e_*(h) = n.$ Thus $(n,h) = (k^*(k(n) \wedge e_*(h)), e(k(n) \wedge e_*(h)))$ and so we conclude that the canonical representatives are precisely the pairs in $\Gl(k^*e_*)$. 

Looking at the prescribed multiplication on $N \times H/E(e,s)$ we see
\begin{align*}
[k^*(g),e(g)]\cdot [k^*(g'),e(g')] 	&= \ophi\inv(g \wedge g') \\
						&= [k^*(g \wedge g'),e(g \wedge g')] \\
						&= [k^*(g) \wedge k^*(g'), e(g) \wedge e(g')].
\end{align*}

This means that when we are dealing with the canonical representations of each class, multiplication is just taking the meet componentwise.

Taken together we see that $(N \times H)/E(e,s) $ is isomorphic to $\Gl(k^*e_*)$ as required.
\end{example}

\section{Failure of the Split Short Five Lemma} \label{splitshortfive}

For any S-protomodular category in the sense of \cite{bourn2015Sprotomodular}, it was shown in that same paper that the split short five lemma holds. This lemma says that when given two split extension $\splitext{N}{k_1}{G_1}{e_1}{s_1}{H}$ and $\splitext{N}{k_2}{G_2}{e_2}{s_2}{H}$, if $\psi \colon G_1 \to G_2$ is a morphism of split extensions, then $\psi$ is an isomorphism. 
Since weakly Schreier extensions are only S-protomodular in a weaker sense \cite{bourn2015partialMaltsev}, the split short five lemma need not hold and in fact does not. In this section we study the morphisms of $\WSExt(H,N)$ before providing a complete characterization of them in \cref{characterization}.

\begin{theorem} \label{theoremsplit}

Let $\splitext{N}{k_1}{G_1}{e_1}{s_1}{H}$ and $\splitext{N}{k_2}{G_2}{e_2}{s_2}{H}$ be two weakly Schreier extensions, let $E_1 = E(e_1,s_1)$ and $E_2 = E(e_2,s_2)$ and let $\psi \colon (N \times H)/E_1 \to (N \times H)/E_2$ be a morphism of the following weakly Schreier extensions. 

\begin{center}
   \begin{tikzpicture}[node distance=2.5cm, auto]
    \node (A) {$N$};
    \node (B) [right of=A] {$(N \times H)/E_1$};
    \node (C) [right of=B] {$H$};
    \node (D) [below of=A] {$N$};
    \node (E) [right of=D] {$(N \times H)/E_2$};
    \node (F) [right of=E] {$H$};
    \draw[normalTail->] (A) to node {$k_1'$} (B);
    \draw[transform canvas={yshift=0.5ex},-normalHead] (B) to node {$e_1'$} (C);
    \draw[transform canvas={yshift=-0.5ex},->] (C) to node {$s_1'$} (B);
    \draw[normalTail->] (D) to node {$k_2'$} (E);
    \draw[transform canvas={yshift=0.5ex},-normalHead] (E) to node {$e_2'$} (F);
    \draw[transform canvas={yshift=-0.5ex},->] (F) to node {$s_2'$} (E);
    \draw[->] (B) to node {$\psi$} (E);
    \draw[double equal sign distance] (A) to (D);
    \draw[double equal sign distance] (C) to (F);
   \end{tikzpicture}
\end{center}

Then $\psi([n,h]_{E_1}) = [n,h]_{E_2}$.
\end{theorem}

\begin{proof}
Since $\psi k_1'(n) = k_2'(n)$, we find that $\psi([n,1]_{E_1}) = [n,1]_{E_2} $. Similarly we find that $\psi([1,h]_{E_1}) = [1,h]_{E_2} $. Now observe that
\begin{align*}
\psi([n,h]_{E_1}) 	&= \psi([n,1]_{E_1}[1,h]_{E_1}) \\
						&= \psi([n,1]_{E_1}) \cdot \psi([1,h]_{E_1}) \\
						&= [n,1]_{E_2}[1,h]_{E_2} \\
						&= [n,h]_{E_2}.
\end{align*}
This completes the proof.
\end{proof}

Notice that this is in agreement with the Schreier case, as there the equivalence classes are all singletons and so \cref{theoremsplit} implies that any morphism between Schreier extensions must be the identity.

From \cref{theoremsplit} we can conclude that any morphism between weakly Schreier extensions must be unique. We thus arrive at the following corollary.

\begin{corollary}
$\WSExt(H,N)$ is a preorder category for all monoids $H$ and $N$.
\end{corollary}

In fact this result can be generalised to any $S$-protomodular category in the sense of \cite{bourn2015partialMaltsev}, as all that is required is that for each extension $\splitext{N}{e}{G}{e}{s}{H}$ in $S$, $(k,s)$ is jointly extremally epic.

Of course \cref{theoremsplit} does not by itself demonstrate that the split short five lemma fails in the case of weakly Schreier extensions. In the following section we will completely determine the conditions which yield a morphism between two weakly Schreier extensions. For now we will demonstrate the failure of the lemma by exhibiting a non-isomorphism between two extensions of Artin glueings.

\begin{example}(Artin glueings)

Let $N$ and $H$ be frames, $f,g \colon H \to N$ meet preserving maps and $\psi \colon \Gl(f) \to \Gl(g)$ a morphism of the following extensions.

\begin{center}
   \begin{tikzpicture}[node distance=1.5cm, auto]
    \node (A) {$N$};
    \node (B) [right of=A] {$\Gl(f)$};
    \node (C) [right of=B] {$H$};
    \node (D) [below of=A] {$N$};
    \node (E) [right of=D] {$\Gl(g)$};
    \node (F) [right of=E] {$H$};
    \draw[normalTail->] (A) to node {$k_1$} (B);
    \draw[transform canvas={yshift=0.5ex},-normalHead] (B) to node {$e_1$} (C);
    \draw[transform canvas={yshift=-0.5ex},->] (C) to node {$s_1$} (B);
    \draw[normalTail->] (D) to node {$k_2$} (E);
    \draw[transform canvas={yshift=0.5ex},-normalHead] (E) to node {$e_2$} (F);
    \draw[transform canvas={yshift=-0.5ex},->] (F) to node {$s_2$} (E);
    \draw[->] (B) to node {$\psi$} (E);
    \draw[double equal sign distance] (A) to (D);
    \draw[double equal sign distance] (C) to (F);
   \end{tikzpicture}
\end{center}

Applying our results in \cite{faul2019artin} to the above diagram we find that $k_1(n) = (n,1)$, $k_2(n) = (n,1)$, $s_1(h) = (f(h),h)$ and $s_2(h) = (g(h),h)$.

Since $\psi$ is a morphism of split extensions we have that $\psi(n,1) = (n,1)$ and $\psi(f(h),h) = (g(h),h)$. This is enough to completely determine $\psi$ as we have

\begin{align*}
\psi(n,h)   &= \psi((n,1) \wedge (f(h),h)) \\
            &= \psi(n,1) \wedge \psi(f(h),h) \\
			&= (n,1) \wedge (g(h),h) \\
			&= (n \wedge g(h),h).	
\end{align*}

However to ensure consistency we require that $(n,1) = \psi(n,1) = (n \wedge g(1),1)$ and that $(g(h),h) = \psi(f(h),h) = (f(h) \wedge g(h),h)$. The former expression will always be true, but the latter is true if and only if $g \le f$.

Thus whenever $g \le f$, we have that $\psi$ as described above will be a morphism of weakly Schreier extensions. Whenever $g$ is strictly less than $f$, it is evident that this $\psi$ is not an isomorphism. For instance let $N$ be a non-trivial frame and let $N = H$. Then take $f$ to be the constant $1$ map and $g$ to be the idenity.
\end{example}

The above example and its consequences are explored in more detail in \cite{faul2019artin}.

\section{Characterizing weakly Schreier extensions}\label{characterization}

In the Schreier case extensions correspond to actions $\alpha \colon H \times N \to N$, which are used to construct a multiplication on the set $N \times H$. If we are to do something similar in the weakly Schreier case, must define a multiplication on some quotient of $N \times H$. The question is then: what are the appropriate quotients to consider and how can we induce the appropriate monoid operations? In what proceeds we will make use of the axiom of choice in the form of \cref{prop:axiom}.

\subsection{The quotient}\label{quotientdef}

Let us tackle the question of the quotients first. We do so by considering the properties of a quotient constructed from a weakly Schreier extension 
as in \cref{weaksemidirectproducts}. 

Here we take inspiration from \cref{sameh,samen,welldefinedaction} and consider only the quotients $Q$ of $N \times H$ whose corresponding equivalence relation satisies the following conditions.

\begin{definition}
Let $E$ be an equivalence relation on $N \times H$. We say $E$ is an \emph{admissible equivalence relation} if it satisfies the following conditions.
\begin{enumerate}
	\item $(n_1,1) \sim (n_2,1)$ implies $n_1 = n_2$,
	\item $(n_1,h_1) \sim (n_2,h_2)$ implies $h_1 = h_2$,
	\item for all $n \in N$, $(n_1,h) \sim (n_2,h)$ implies $(nn_1,h) \sim (nn_2,h)$ and
	\item for all $h' \in H$, $(n_1,h) \sim (n_2,h)$ implies $(n_1,hh') \sim (n_2,hh')$.
\end{enumerate}

We call the induced quotient $N \times H/E$ an \emph{admissible quotient} if $E$ is admissible. 
\end{definition}

Given an admissible quotient $Q$ on $N \times H$ we write $(n,h) \sim_Q (n',h)$ to mean that $(n,h)$ belongs to the same equivalence class as $(n',h)$.

Notice that when $h$ has a right inverse, condition (1) and (4) together imply the following.

\begin{proposition}\label{inverses}
Let $Q$ be an admissible quotient on $N \times H$. If $h \in H$ has a right inverse, then $(n,h) \sim (n',h)$ implies $n = n'$.
\end{proposition}
 
In particular this means that for groups, the only admissible quotient will be the discrete one. This is consistent with the observation that all split extensions of groups are Schreier.

We now consider an action of $N$ on $Q$ and an action of $H$ on $Q$, which will be well-defined thanks to conditions (3) and (4) above. For each $n' \in N$ let $n' \ast [n,h] = [n'n,h]$ and for each $h' \in H$ let $[n,h] \ast h' = [n,hh']$. Equipped with these actions we can consider $Q$ to be similar in character to a bi-module.

\subsection{The multiplication}

Suppose we have a weakly Schreier extension $\splitext{N}{k}{G}{e}{s}{H}$ and let $q \colon G \to N$ be an associated Schreier retraction. Let us construct the associated weak semidirect product $(N \times H)/E(e,s)$ as in \cref{weaksemidirectproducts} and examine the multiplication in more detail. In this section we make extensive use of (3) in \cref{prop:reviewersuggestion}, which says that $kq(s(h)k(n)) \cdot s(h) = s(h)k(n)$.

\begin{proposition}\label{startweak}
Let $\splitext{N}{k}{G}{e}{s}{H}$ be a weakly Schreier extension and let $q$ be an associated Schreier retraction. For $[n_1,h_1], [n_2,h_2] \in (N \times H)/E(e,s)$ we can equivalently express the multiplication as \[[n_1,h_1] \cdot [n_2,h_2] = n_1 \ast [q(s(h_1)k(n_2)),h_1] \ast h_2.\] 	
\end{proposition}

\begin{proof}
We must show that $n_1 \ast [q(s(h_1)k(n_2)),h_1] \ast h_2$ is sent by $\ophi$ to $k(n_1) \cdot s(h_1) \cdot k(n_2) \cdot s(h_2).$

We know that \[\ophi(n_1 \cdot q(s(h_1)k(n_2)),h_1h_2) = k(n_1) \cdot kq(s(h_1)k(n_2)) \cdot s(h_1) \cdot s(h_2).\]

Since $kq(s(h)k(n)) \cdot s(h) = s(h) \cdot k(n)$, the above expression simplifies to $k(n_1) \cdot s(h_1) \cdot k(n_2) \cdot s(h_2)$ as required.
\end{proof}

This presentation of the multiplication suggests that something resembling the actions of the Schreier case will play an equally crucial role in defining the multiplication on $Q$. 

Note that since the multiplication of $(N \times H)/E(e,s)$ was defined \Cref{weaksemidirectproducts} in  without any reference to Schreier retractions, it must be that all Schreier retractions induce the same multiplication.

\begin{corollary}\label{cor:chooseretraction}
Let $\splitext{N}{k}{G}{e}{s}{H}$ be a weakly Schreier extension and let $q$ and $q'$ be Schreier retractions. Then $[n_1q(s(h_1)k(n_2)),h_1h_2] = [n_1q'(s(h_1)k(n_2)),h_1h_2]$ for all $n_1,n_2 \in N$ and $h_1,h_2 \in H$.
\end{corollary}

Now let $\alpha \colon H \times N \to N$ send $(h,n)$ to $q(s(h)k(n))$ and let us study its properties.

\begin{proposition}\label{compatability1}
Let $\splitext{N}{k}{G}{e}{s}{H}$ be a weakly Schreier extension, $q$ an associated Schreier retraction and let $\alpha(b,a) = q(s(b)k(a))$. If $(n_1,h) \sim (n_2,h)$, then $n_1 \ast [\alpha(h,n),h] = n_2 \ast [\alpha(h,n),h]$.
\end{proposition}
 
\begin{proof}
Suppose that $(n_1,h) \sim (n_2,h)$ and consider the following calculation.

\begin{align*}
\ophi(n_1 \ast [\alpha(h,n),h]) &= k(n_1) \cdot k\alpha(h,n) \cdot s(h)\\
&= k(n_1) \cdot s(h) \cdot k(n)\\
&= k(n_2) \cdot s(h) \cdot k(n)\\
&= k(n_2) \cdot k\alpha(h,n) \cdot s(h)\\
&= \ophi(n_2 \ast [\alpha(h,n),h])
\end{align*} 

Since $\ophi$ is injective, we must have that $n_1 \ast [\alpha(h,n),h] = n_2 \ast [\alpha(h,n),h]$.
\end{proof}

Similarly we have the following proposition.

\begin{proposition}\label{compatability2}
Let $\splitext{N}{k}{G}{e}{s}{H}$ be a weakly Schreier extension, $q$ an associated Schreier retraction and let $\alpha(b,a) = q(s(b)k(a))$. If $(n,h') \sim (n',h')$, then $[\alpha(h,n),h] \ast h' = [\alpha(h,n'),h] \ast h'$.	
\end{proposition}

%
%
%

Given an admissible quotient $Q$, any maps $\alpha \colon H \times N \to N$ satisfying \cref{compatability1} and \cref{compatability2} we call \emph{pre-actions compatible with $Q$}.

Next we show that $\alpha$ satisfies conditions analogous to being an action in the Schreier case.

\begin{proposition}\label{multiplicativelaw}
Let $\splitext{N}{k}{G}{e}{s}{H}$ be a weakly Schreier extension, $q$ an associated Schreier retraction and let $\alpha(b,a) = q(s(b)k(a))$. Then $[\alpha(h,nn'),h] = [\alpha(h,n) \cdot \alpha(h,n'),h]$.
\end{proposition}

\begin{proof}
In order to prove that these classes are equal we show that $\ophi$ maps them to the same element of $G$. Thus consider
\begin{align*}
\ophi([\alpha(h,nn'),h]) &= kq(s(h)k(nn')) \cdot s(h) \\
&= s(h) \cdot k(n) \cdot k(n').	
\end{align*}
We also have
\begin{align*}
\ophi([\alpha(h,n)\alpha(h,n'),h]) &= kq(s(h)k(n)) \cdot kq(s(h)k(n')) \cdot s(h) \\
&= kq(s(h)k(n)) \cdot s(h) \cdot k(n') \\
&= s(h) \cdot k(n) \cdot k(n').	
\end{align*}
As discussed above, this gives that $[\alpha(h,nn'),h] = [\alpha(h,n) \cdot \alpha(h,n'),h]$.
\end{proof}

\begin{proposition}\label{compositionallaw}
Let $\splitext{N}{k}{G}{e}{s}{H}$ be a weakly Schreier extension, $q$ an associated Schreier retraction and let $\alpha(b,a) = q(s(b)k(a))$. Then $[\alpha(hh',n),hh'] = [\alpha(h,\alpha(h',n)),hh']$.
\end{proposition}

\begin{proof}
We need only show that $\ophi$ maps each class to the same element. We have
\begin{align*}
\ophi([\alpha(hh',n),hh']) &= kq(s(hh')k(n)) \cdot s(hh')\\
&= s(h) \cdot s(h') \cdot k(n).
\end{align*}

Compare it to the following.
\begin{align*}
\ophi([\alpha(h,\alpha(h',n)),hh']) &= kq(s(h)kq(s(h')k(n))) \cdot s(h) \cdot s(h') \\
&= s(h) \cdot kq(s(h')k(n)) \cdot s(h') \\
&= s(h) \cdot s(h') \cdot k(n)
\end{align*}

Thus $\ophi$ sends $[\alpha(hh',n),hh']$ and $[\alpha(h,\alpha(h',n)),hh']$ to the same element and so they are equal.
\end{proof}

\begin{proposition}
Let $\splitext{N}{k}{G}{e}{s}{H}$ be a weakly Schreier extension, $q$ an associated Schreier retraction and let $\alpha(b,a) = q(s(b)k(a))$. Then $[\alpha(h,1),h] = [1,h]$.
\end{proposition}

\begin{proof}
Just consider $\ophi([\alpha(h,1),h] = kq(s(h))s(h) = s(h) = \ophi([1,h])$ and observe that $\ophi$ is injective.
\end{proof}

The following result is proved similarly.

\begin{proposition}
Let $\splitext{N}{k}{G}{e}{s}{H}$ be a weakly Schreier extension, $q$ an associated Schreier retraction and let $\alpha(b,a) = q(s(b)k(a))$. Then $[\alpha(1,n),1] = [n,1]$.
\end{proposition}


We now say an action is a compatible pre-action which satisfies the above properties.

\begin{definition}\label{actiondef}
A function $\alpha \colon H \times N \to N$ is an \emph{action} with respect to an admissible quotient $Q$ of $N \times H$ if it satisfies the following conditions:

\begin{enumerate}
	\item $(n_1,h) \sim (n_2,h)$ implies $n_1 \ast [\alpha(h,n),h] = n_2 \ast [\alpha(h,n),h]$ for all $n \in N$,
	\item $(n,h') \sim (n',h')$ implies $[\alpha(h,n),h] \ast h' = [\alpha(h,n'),h] \ast h'$ for all $h \in H$,
	\item $[\alpha(h,nn'),h] = [\alpha(h,n)\cdot\alpha(h,n'),h]$,
	\item $[\alpha(hh',n),hh'] = [\alpha(h,\alpha(h',n)),hh']$,
	\item $[\alpha(h,1),h] = [1,h]$,
	\item $[\alpha(1,n),1] = [n,1]$.
\end{enumerate}
	
\end{definition}

Let us again draw attention to this definition as it applies to groups. By \cref{inverses}, the only admissible quotient is discrete and so conditions (1) and (2) are immediately satisfied. The remaining four conditions then reduce to requiring that $\alpha$ be an action in the traditional sense. This same argument gives that $\alpha$ must be an action in the Schreier setting.

Let $\mathrm{Act}_Q$ denote the set of actions with respect to $Q$. Given a quotient $Q$ and an action $\alpha \in \mathrm{Act}_Q$ we can equip $Q$ with a multiplication as follows.
\[[n,h] \cdot [n',h'] = n \ast [\alpha(h,n'),h] \ast h'.\]

Let us check that this is well defined. Suppose that $[a,h] = [n,h]$ and that $[a',h'] = [n',h']$.

Then by condition 1 of \cref{actiondef} we have that $[a,h][n',h'] = [n,h][n',h']$. Then applying condition 2 of \cref{actiondef} we get that $[a,h][a',h'] = [a,h][n',h']$. Thus this operation is well defined and it remains to prove it is associative and has an identity.

\begin{proposition}
Let $Q$ be an admissible quotient of $N \times H$ and $\alpha \in \mathrm{Act}_Q$. Then $[n,h][n',h'] = n \ast [\alpha(h,n'),h] \ast h'$ makes $Q$ a monoid with identity $[1,1]$.	
\end{proposition}

\begin{proof}
For the identity simply observe $[n,h][1,1] = n \ast [\alpha(h,1),h] \ast 1 = n \ast [1,h] = [n,h]$ and then $[1,1][n,h] = 1 \ast [\alpha(1,n),1] \ast h = [n,1] \ast h = [n,h]$.

For associativity consider
\begin{align*}
([n_1,h_1][n_2,h_2])[n_3,h_3] 
&= [n_1 \cdot \alpha(h_1,n_2), h_1h_2][n_3,h_3]\\
&= 	[n_1 \cdot \alpha(h_1,n_2) \cdot \alpha(h_1h_2,n_3),h_1h_2h_3] \\
&= n_1 \alpha(h_1,n_2) \ast [\alpha(h_1h_2,n_3),h_1h_2] \ast h_3 \\
&= n_1 \alpha(h_1,n_2) \ast [\alpha(h_1, \alpha(h_2,n_3)), h_1h_2] \ast h_3
\end{align*}

and compare it to
\begin{align*}
[n_1,h_1]([n_2,h_2][n_3,h_3]) &= [n_1,h_1][n_2 \cdot \alpha(h_2,n_3), h_2h_3]\\
&= [n_1 \cdot \alpha(h_1,n_2 \cdot \alpha(h_2,n_3)),h_1h_2h_3]\\
&= n_1 \ast [\alpha(h_1, n_2\alpha(h_2,n_3)),h_1]\ast h_2h_3\\
&= n_1 \ast [\alpha(h_1,n_2) \cdot \alpha(h_1,\alpha(h_2,n_3),h_1] \ast h_2h_3\\
&= n_1\alpha(h_1,n_2) \ast [\alpha(h_1,\alpha(h_2,n_3)),h_1h_2] \ast h_3.
\end{align*}

Thus this operation is indeed associative and so $Q$ becomes a monoid. 	
\end{proof}

Let us call the resulting monoid $Q_{\alpha}$.

We may now state the result that justifies our use of weak semidirect product.

\begin{theorem}\label{canon}
The diagram $\splitext{N}{k}{Q_{\alpha}}{e}{s}{H}$ where $k(n) = [n,1]$, $e([n,h]) = h$ and $s(h) = [1,h]$, is a weakly Schreier extension.
\end{theorem}

\begin{proof}
Observe that for every element $[n,h]$ we can write $k(n)s(h) = [n,1][1,h] = n \ast [\alpha(1,1),1] \ast h = n \ast [1,1] \ast h = [n,h]$. So it satisfies the weakly Schreier condition. Thus it remains only to show that $k$ is the kernel of $e$, and $e$ the cokernel of $k$.

It is clear that $ek = 0$ and further that the image of $k$ is precisely the submonoid sent to $1$. Thus for any $t \colon X \to G$ satisfying $et=0$, it must map into the image $k$. It is clear then that $t$ factors uniquely through $k$. 

Next suppose we have a map $t \colon Q \to X$ satisfying $tk = 0$. Consider $t([n,h]) = t([n,1][1,h]) = t([1,h]).$ Thus where $t$ sends a class is entirely determined by where it sends $[1,h]$. Thus define $t' \colon H \to X$ which sends $h$ to $t([1,h])$. It is clear that $t = t'e$ and is unique as $e$ is epic.
\end{proof}

\begin{proposition}
Two actions $\alpha, \alpha' \in \mathrm{Act}_Q$ induce isomorphic weakly Schreier extensions on $Q$ if and only if for all $(h,n) \in H \times N$ we have $[\alpha(h,n),h] = [\alpha'(h,n),h]$.
\end{proposition}

\begin{proof}
Suppose that $\alpha$ and $\alpha'$ satisfy that $[\alpha(h,n),h] = [\alpha'(h,n),h]$ and let $[n,h]\cdot_{\alpha}[n',h']$ and $[n,h]\cdot_{\alpha'}[n',h']$ denote the multiplication in $Q_{\alpha}$ and $Q_{\alpha'}$ respectively. Then observe that $[n,h]\cdot_{\alpha}[n',h'] = n\ast [\alpha(h,n'),h] \ast h' = n \ast [\alpha'(h,n'),h] \ast h' = [n,h]\cdot_{\alpha'}[n',h']$.

For the other direction suppose there exists a pair $(h,n) \in H \times N$ such that $[\alpha(h,n),h] \ne [\alpha'(h,n),h]$ and consider the associated weakly Schreier extensions $\splitext{N}{k}{Q_{\alpha}}{e}{s}{H}$ and $\splitext{N}{k'}{Q_{\alpha'}}{e'}{s'}{H}$. Suppose we have an isomorphism of extensions $\psi \colon Q_{\alpha} \to Q_{\alpha'}$. By \cref{theoremsplit} we know that $\psi([n,h]) = [n,h]$, but observe that $\psi([\alpha(h,n),h]) = \psi([1,h]\cdot_{\alpha}[n,1]) = [1,h]\cdot_{\alpha'}[n,1] = [\alpha'(h,n),h] \ne [\alpha(h,n),h]$. This yields a contradiction.
\end{proof}

Let us take the quotient of the set of actions by the equivalence relation given by $\alpha \sim \alpha'$ if and only if $[\alpha(h,n),h] = [\alpha'(h,n),h]$ for all $n \in N$ and $h \in H$. Call this set $\mathrm{Act}_Q/\hspace{-1mm}\sim$. We then have a process which transforms a pair $(Q,[\alpha])$, where $Q$ is an admissible quotient of $N \times H$ and $[\alpha]$ an equivalence class of actions relative to $Q$, into a weakly Schreier extension. 

\begin{proposition}\label{twelldefined}
Let $\splitext{N}{k}{G}{e}{s}{H}$ be weakly Schreier and let $q$ and $q'$ be associated Schreier retractions. Then if $\alpha(h,n) = q(s(h)k(n))$ and $\alpha'(h,n) = q'(s(h)k(n))$ we have that $\alpha \sim \alpha'$ in $\mathrm{Act}_Q/\hspace{-1mm}\sim$.
\end{proposition}

\begin{proof}
This follows from \cref{cor:chooseretraction}. Simply consider the result of $[n,1][1,h]$ in $(N \times H)/E(e,s)$.
\end{proof}

In order to show that the above forms a complete characterization of the weakly Schreier extensions between $H$ and $N$ and to simultaneously characterise the morphisms of $\WSExt(H,N)$, we introduce the following preorder.

\begin{definition}
Let $\mathrm{WAct}(H,N)$ be the preorder whose objects are pairs $(Q,[\alpha])$ where $Q$ is an admissible quotient on $N \times H$ and $[\alpha]\in \mathrm{Act}_Q/\hspace{-1mm}\sim$. We say that $(Q,[\alpha]) \le (Q',[\alpha'])$ if and only if $(n,h) \sim_Q (n',h)$ implies that $(n,h) \sim_{Q'} (n',h)$ and $(\alpha(h,n),h) \sim_{Q'} (\alpha'(h,n),h)$.
\end{definition}

The relationship $(Q,[\alpha]) \le (Q',[\alpha'])$ should be thought of as saying that $Q$ is a finer quotient than $Q'$ and that $[\alpha]$ agrees with $[\alpha']$ on $Q'$. 

\begin{theorem}
The categories $\mathrm{WAct}(H,N)$ and $\WSExt(H,N)$ are equivalent.
\end{theorem}

\begin{proof}
Let $S \colon \mathrm{WAct}(H,N) \to \WSExt(H,N)$ send the pair $(Q,[\alpha])$ to $\splitext{N}{k}{Q_\alpha}{e}{s}{H}$, described in \cref{canon}. Let us begin by demonstrating that $S$ preserves the order. 

Suppose that $(Q,[\alpha]) \leq (Q',[\alpha'])$. \Cref{theoremsplit} tells us that any morphism of split extensions between $S(Q,[\alpha])$ and $S(Q',[\alpha])$ must be a map $\psi \colon Q_\alpha \to Q'_{\alpha'}$ sending $[n,h]_Q$ to $[n,h]_{Q'}$. Let us show that indeed $\psi$ is a well-defined morphism of split extensions. 

Suppose that $[n_1,h]_Q = [n_2,h]_Q$. We know that $(Q,[\alpha]) \leq (Q',[\alpha'])$, and so we have that $(n_1,h) \sim_Q (n_2,h)$ implies that $(n_1,h) \sim_{Q'} (n_2,h)$. Thus $ \psi([n_1,h]_Q) = [n_1,h]_{Q'} = [n_2,h]_{Q'} = \psi([n_2,h]_Q)$, which proves that our description of $\psi$ is well-defined. 

Next we need that $\psi$ preserves the operation. In order to prove this we make use of the fact that $(Q,[\alpha]) \leq (Q',[\alpha'])$ implies that $(\alpha(h,n),h) \sim_{Q'} (\alpha'(h,n),h)$ in the following calculation. 

\begin{align*}
\psi([n,h]_Q \cdot [n',h']) 	&= \psi([n \cdot \alpha(h,n'),hh']_Q) \\
									&= [n \cdot \alpha(h,n'),hh']_{Q'} \\
									&= n \ast [\alpha(h,n'),h]_{Q'} \ast h' \\
									&= n \ast [\alpha'(h,n'),h]_{Q'} \ast h' \\
									&= [n,h]_{Q'} \cdot [n',h']_{Q'} 
\end{align*}

It is apparent that $\psi$ makes the required diagram commute and so is a morphism of split extensions. Thus $S$ preserves the order as required.

Let $\splitext{N}{k}{G}{e}{s}{H}$ be a weakly Schreier extension, $E = E(e,s)$ and let $q$ be an associated Schreier retraction. Then let $T \colon \WSExt(H,N) \to \mathrm{WAct}(H,N)$ send $\splitext{N}{k}{G}{e}{s}{H}$ to $((N \times H)/E,[\alpha])$ where $\alpha(h,n) = q(s(h)k(n))$. By \cref{twelldefined} we have that $T$ is well defined. We must show that $T$ respects the preorder structure.

Suppose we have a morphism $\psi \colon G_1 \to G_2$ of weakly Schreier extensions as in the following diagram.

\begin{center}
   \begin{tikzpicture}[node distance=1.5cm, auto]
    \node (A) {$N$};
    \node (B) [right of=A] {$G_1$};
    \node (C) [right of=B] {$H$};
    \node (D) [below of=A] {$N$};
    \node (E) [right of=D] {$G_2$};
    \node (F) [right of=E] {$H$};
    \draw[normalTail->] (A) to node {$k_1$} (B);
    \draw[transform canvas={yshift=0.5ex},-normalHead] (B) to node {$e_1$} (C);
    \draw[transform canvas={yshift=-0.5ex},->] (C) to node {$s_1$} (B);
    \draw[normalTail->] (D) to node {$k_2$} (E);
    \draw[transform canvas={yshift=0.5ex},-normalHead] (E) to node {$e_2$} (F);
    \draw[transform canvas={yshift=-0.5ex},->] (F) to node {$s_2$} (E);
    \draw[->] (B) to node {$\psi$} (E);
    \draw[double equal sign distance] (A) to (D);
    \draw[double equal sign distance] (C) to (F);
   \end{tikzpicture}
\end{center}

Let $q_1$ be an associated Schreier retraction of $\splitext{N}{k_1}{G_1}{e_1}{s_1}{H}$ and $q_2$ an associated Schreier retraction of $\splitext{N}{k_2}{G_2}{e_2}{s_2}{H}$. Further let $E_1 = E(e_1,s_2)$ and $E_2 = E(e_2,s_2)$.

Then in order to show that $T$ is order preserving we must show that $(N \times H/E(e_1,s_1),[\alpha_1]) \leq (N \times H/E(e_2,s_2), [\alpha_2])$. This requires us to show that if $(n,h) \sim_{E_1} (n',h)$ then $(n,h) \sim_{E_2} (n',h)$ and finally that $(\alpha_1(h,n),h) \sim_{E_2} (\alpha_2(h,n),h)$.

Suppose that $(n,h) \sim_{E_1} (n',h)$. This means that $k_1(n) \cdot s_1(h) = k_1(n') \cdot s_1(h)$. In order to show that $(n,h)$ and $(n',h)$ are related in $E_2$, we must show that $k_2(n) \cdot s_2(h) = k_2(n') \cdot s_2(h)$. Consider 
\begin{align*}
k_2(n) \cdot s_2(h)	&= \psi k_1(n) \cdot \psi s_1(h) \\
						&= \psi(k_1(n)s_1(h)) \\
						&= \psi(k_1(n')s_1(h)) \\
						&= k_2(n') \cdot s_2(h).
\end{align*}

In order to show that the second condition holds consider the following calculation.
\begin{align*}
k_2\alpha_1(h,n) \cdot s_2(h)	&= k_2q_1(s_1(h)k_1(n)) \cdot s_2(h) \\
									&= \psi k_1q_1(s_1(h)k_1(n)) \cdot \psi s_1(h) \\
									&= \psi(k_1q_1(s_1(h)k_1(n))s_1(h)) \\
									&= \psi(s_1(h)k_1(n)) \\
									&= s_2(h)k_2(n) \\
									&= k_2q_2(s_2(h)k_2(n)) \cdot s_2(h) \\
									&= k_2\alpha_2(h,n) \cdot s_2(h)
\end{align*}

Thus indeed $((N \times H)/E(e_1,s_1),[\alpha_1]) \leq ((N \times H)/E(e_2,s_2), [\alpha_2])$ and so $T$ preserves the order.

Finally we now show that the functors $T$ and $S$ form an equivalence of categories.

\Cref{ophiiso} and  \cref{startweak} together give us that $ST$ is equivalent to the identity and so we can shift our attention to $TS$.

Suppose we apply $S$ to a pair $(Q,[\alpha])$ and generate the extension $\splitext{N}{k}{Q_\alpha}{e}{s}{H}$. Let $q$ be an associated Schreier retraction. Let us thus define a map $\alpha'(h,n) = q(s(h)k(n))$. So $TS(Q,[\alpha])$ yields the pair $(N \times H/E(e,s), [\alpha'])$. We will now show that $(Q,[\alpha]) = (N \times H/E(e,s), [\alpha'])$.

The pairs $(n,h) \sim (n',h)$ related in $Q$ are precisely those pairs satisfying 
\begin{align*}
k(n) \cdot s(h) &= [n,1] \cdot [1,h] \\ 
                &= [n,h] \\
                &= [n',h] \\ 
                &= k(n') \cdot s(h). 
\end{align*}
These in turn are precisely the pairs $(n,h) \sim (n',h)$ related in $E(e,s)$. Thus we get that $Q = (N \times H)/E(e,s)$.

Per our definition of $\alpha'$ we must show that $[\alpha'(h,n),h] = [\alpha(h,n),h]$ --- that is, $kq(s(h)k(n)) \cdot s(h) = k\alpha(h,n) \cdot s(h)$. Consider

\begin{align*}
kq(s(h)k(n)) \cdot s(h)         &= s(h) \cdot k(n) \\
								&= [\alpha(h,n),h] \\ 
								&= [\alpha(h,n),1] \cdot [1,h] \\
								&= k\alpha(h,n) \cdot s(h).
\end{align*}
 
Thus $[\alpha] = [\alpha']$, which then finally gives that $TS$ is the identity.
\end{proof}

We thus have a full characterization of all weakly Schreier extensions in the category of monoids, as well as a characterization of the morphisms between them given by the following corollary.

\begin{corollary}
Let $\splitext{N}{k_1}{G_1}{e_1}{s_1}{H}$ and $\splitext{N}{k_2}{G_2}{e_2}{s_2}{H}$ be weakly Schreier extensions, $q_1$ and $q_2$ respective associated Schreier retractions and let $E_1 = E(e_1,s_1)$ and $E_2 = E(e_2,s_2)$. Then a morphism $\psi \colon G_1 \to G_2$ of split extensions exists if and only if for all $n\in N$ and $h\in H$ we have $(n,h) \sim_{E_1} (n',h)$ implies that $(n,h) \sim_{E_2} (n',h)$ and $(q_1(s(h)k(n)),h) \sim_{E_2} (q_2(s(h)k(n)),h)$. 	
\end{corollary}

From this characterization, we can deduce the following results about weakly Schreier extensions in the full subcategories of commutative monoids and abelian groups.

\begin{proposition}
Let $\splitext{N}{k}{G}{e}{s}{H}$ be a weakly Schreier extension in which $N$, $G$ and $H$ are commutative. Then if $(Q,[\alpha])$ corresponds to $\splitext{N}{k}{G}{e}{s}{H}$ it must be that $(\alpha(h,n),h) \sim (n,h)$.
\end{proposition}

\begin{proof}
If $(Q,[\alpha])$ corresponds to $\splitext{N}{k}{G}{e}{s}{H}$ then $Q_\alpha$ is isomorphic to $G$. Thus $Q_\alpha$ is commutative and so $[n,h] = [n,1] \cdot [1,h] = [1,h] \cdot [n,1] = [\alpha(h,n),h]$. 
\end{proof}

Thus $\alpha$ is equivalent to the trivial action and so multiplication in $Q_\alpha$ is given by $[n,h] \cdot [n',h'] = [nn',hh']$. This is of course in agreement with our understanding of the Artin glueing case.

In the subcategory of abelian groups where all admissible quotients must be discrete we then find that the only weakly Schreier extension is the direct product.

\section{Constructing examples}

In this section we concern ourselves with the practicalities of constructing weakly Schreier extensions. 

\subsection{Generalising the Artin glueing}

We present a construction reminiscent of the Artin glueing, for weakly Schreier extensions of $H$ by $N$, where $N$ is commutative.

\begin{proposition}\label{likeartin}
Let $N$ be a commutative monoid and $f \colon H \to N$ a monoid homomorphism. Then the equivalence relation $E$ on $N \times H$ given by $(n,h) \sim (n',h')$ if and only if $n \cdot f(h) = n' \cdot f(h')$ and $h = h'$, is admissible. 
\end{proposition}

\begin{proof}
By definition we have that $(n,1) \sim (n',1)$ implies that $n = n'$ and that $(n,h) \sim (n',h')$ implies that $h = h'$. If $(n_1,h) \sim (n_2,h)$ then we get $n_1 \cdot f(h) = n_2 \cdot f(h)$. Multiplying both sides on the left by $n$ yields $n \cdot n_1 \cdot f(h) = n \cdot n_2 \cdot f(h)$ which then gives that $(nn_1,h) \sim (nn_2,h)$ for all $n \in N$. If instead we multiplied both sides of the equation on the right by $f(h')$ and use that $f$ is a monoid homomorphism, we see that $(n_1,hh') \sim (n_2,hh')$ for all $h'$. Thus $E$ induces an admissible quotient as required.
\end{proof}

\begin{proposition}\label{actionartin}
Let $N$ be a commutative monoid. The trivial action $\alpha(h,n) = n$ is compatible with $(N \times H)/E$ for $E$ taken from \cref{likeartin}.
\end{proposition}

\begin{proof} 
Consider the preaction $\alpha(h,n) = n$ for all $h\in H$ and $n \in N$. Were this a compatible action on $N \times H/E$ it would yield multiplication $[n,h][n',h'] = [n \cdot \alpha(h,n'),hh'] = [nn',hh']$. Let us now show that $\alpha$ is a compatible action which entails showing that it satisfies the six condition in \cref{actiondef}. 

First suppose that $(n_1,h) \sim (n_2,h)$ and consider 
$(n_1 \cdot \alpha(h,n),h)$ and $(n_2 \cdot \alpha(h,n),h)$. In order to show that they are related in $E$ we consider the following. 
\begin{align*}
n_1 \cdot \alpha(h,n) \cdot f(h)    &= n_1n\cdot f(h) \\ 
                                    &= nn_1 \cdot f(h) \\
                                    &= nn_2\cdot f(h) \\ 
                                    &= n_2 \alpha(h,n) \cdot f(h)
\end{align*}

Similarly if we let $(n,h') \sim (n',h')$ we can consider $(\alpha(h,n),hh')$ and $(\alpha(h,n'),hh')$. We perform a similar calculation 
\begin{align*}
\alpha(h,n) \cdot f(h) \cdot f(h')  &= n \cdot f(h) \cdot f(h') \\ 
                                    &= n \cdot f(h') \cdot f(h) \\ 
                                    &= n' \cdot f(h') \cdot f(h) \\ 
                                    &= \alpha(h,n') \cdot f(h) \cdot f(h'),
\end{align*}
and see that indeed $(\alpha(h,n),hh') \sim (\alpha(h,n'),hh')$. Thus $\alpha$ is a compatible preaction.
 
The final four conditions follow immediately from the definition of $\alpha$.
\end{proof}

Notice that two monoid homomorphisms $f,g \colon H \to N$ may yield the same weakly Schreier extension. In fact when $N$ is right-cancelative, each homomorphism $f \colon H \to N$ yields the usual product $N \times H$.

\begin{example}
Let $N$ be a meet-semilattice and $f \colon H \to N$ a monoid map. Then $\splitext{N}{k}{\Gl(f)}{e}{s}{H}$ is a weakly Schreier extension, where $\Gl(f)$ is the set of pairs $(n,h)$ in which $n \leq f(h)$ with $(n,h) \cdot (n',h')=(n \wedge n', h \cdot h')$, $k(n) = (n,1)$, $e(n,h) = h$ and $s(h) = (f(h),h)$. 

To see this consider the admissible quotient given by $E$ as in \cref{likeartin}. If we assume $H$ is equipped with the discrete order, then each equivalence class $[n,h]$ has a smallest element given by $(n \wedge f(h), h)$. The set of these representatives is easily seen to be $\Gl(f)$.

\Cref{actionartin} gives that the action $\alpha(h,n) = n$ for all $n \in N$ and $h \in H$ is compatible with $N \times H/E$. This induces multiplication $[n,h]\cdot [n',h'] = [n \wedge n', h \cdot h']$. If $(n,h), (n',h') \in \Gl(f)$ then $(n \wedge n', h \cdot h') \in \Gl(f)$. This gives that $\alpha$ induced componentwise multiplication on $\Gl(f)$ and so we are done. 
\end{example}

\subsection{The coarsest admissible quotient}

Schreier extensions of $H$ by $N$ may be thought of as the weakly Schreier extensions with the finest admissible quotient on $N \times H$. As discussed above, \cite{martins2013semidirect} provides a complete characterization of all actions compatible with this discrete quotient. Dual to this problem might be considering the coarsest admissible quotient and characterizing the actions compatible with it. 

 As discussed in \cref{inverses}, if $h \in H$ has a right inverse, then $(n,h) \sim (n',h)$ implies that $n = n'$. Taking this and the fact that $(n,h) \sim (n',h')$ must imply that $h = h'$ as our only constraints, we can consider the equivalence relation generated by the condition $(n,h) \sim (n',h)$ iff $h$ has no right inverse. If admissible, this would be the coarsest admissible quotient on $N \times H$.
 
 
 Let $L(H) \subseteq H$ be the submonoid of right invertible elements --- that is, elements $h$ with right inverses and let $\overline{L(H)}$ be the the set of elements which are not right invertible.

\begin{proposition}
Let $(Q,[\alpha]) \in \mathrm{WAct}(H,N)$. Then $\alpha\hspace{-0.8mm}\mid_{L(H) \times N}$ is an action of $L(H)$ on $N$.
\end{proposition}

\begin{proof}
\Cref{inverses} gives us that for all $x\in L(H)$, $(n,x) \sim (n',x)$ implies that $n = n'$. Thus $[\alpha(x,nn'),x] = [\alpha(x,n)\alpha(x,n'),x]$ implies that $\alpha(x,nn') = \alpha(x,n)\alpha(x,n')$. 

Similarly if $x,x' \in L(H)$ we get that $[\alpha(xx',n),xx'] = [\alpha(x,\alpha(x',n)),xx']$ implies $\alpha(xx',n) = \alpha(x,\alpha(x',n))$, as $xx' \in L(H)$.

Finally it is easy to see that these same arguments give that $\alpha(1,n) = n$ and $\alpha(x,1) = 1$.
\end{proof}

This result tells us that in general there are some maps $\alpha \colon H \times N \to N$ such that no admissible quotient makes it an action.

We also get the following corollary.

\begin{corollary}
Let $\splitext{N}{k}{G}{e}{s}{H}$ be a weakly Schreier extension and $(N \times H)/E(e,s)$ the associated weak semidirect product. Then $(N \times L(H))/E(e,s) \subseteq (N \times H)/E(e,s)$ is a submonoid which is itself part of a Schreier extension of $L(H)$ by $N$ given by $(N \times L(H),[\alpha\hspace{-0.8mm}\mid_{L(H) \times N}])$.
\end{corollary}
 
Thus in each weakly Schreier extension we understand how a particular submonoid behaves. Let us now study how the rest of the monoid behaves. 

\begin{proposition}
Let $H$ be a monoid. The subset $\overline{L(H)} \subseteq H$ is a monoid right-ideal --- that is, if $y \in \overline{L(H)}$ and $x \in H$ then $yx \in \overline{L(H)}$.	
\end{proposition}

\begin{proof}
Supose $y \in \overline{L(H)}$ and let $x \in H$. Then if $yx(yx)^* = 1$, that would imply that $x(yx)^*$ was a right inverse for $y$, contradicting the fact that $y \in \overline{L(H)}$.
\end{proof}

\begin{corollary}
Let $\splitext{N}{k}{G}{e}{s}{H}$ be a weakly Schreier extension and $(N \times H)/E(e,s)$ the associated weak semidirect product. Then the subset of $(N \times H)/E(e,s)$ comprising the classes $[n,y]$ where $y \in \overline{L(H)}$, is a right ideal of $(N \times H)/E(e,s)$.
\end{corollary}

\begin{definition}
Let $E$ be the equivalence relation given by $(n,h) \sim (n',h)$ implies $n = n'$ whenever $h \in L(H)$ and $(n,h) \sim (n',h)$ for all $n,n' \in N$ whenever $h \in \overline{L(H)}$. We call the quotient $Q$ induced by $E$ the \emph{coarse quotient} on $N \times H$.	
\end{definition}

\begin{proposition}\label{coarseadmis}
The coarse quotient $Q$ on $N \times H$ is admissible.
\end{proposition}

\begin{proof}

Since $1 \in L(H)$ we immediately get that $(n,1) \sim (n',1)$ implies that $n=n'$. It is also not hard to see that $(n,h) \sim (n',h')$ implies that $h=h'$ irrespective of where $h$ and $h'$ come from.

Now suppose $(n_1,h) \sim (n_2,h)$. If $h \in L(H)$ then $n_1=n_2$ and so for all $n \in N$ and $h' \in H$ we have $(nn_1,h) = (nn_2,h)$ and $(n_1,hh') = (n_2,hh')$.  

If instead $h \in \overline{L(H)}$ it is immediate that for all $n \in N$, $(nn_1,h) \sim (nn_2,h)$. For $h' \in H$ we have that $hh' \in \overline{L(H)}$ as $\overline{L(H)}$ is a right ideal and so $(n_1,hh') \sim (n_2,hh')$.
\end{proof}

\begin{proposition}
Suppose that $\overline{L(H)}$ is a two-sided ideal. Each map $\alpha \colon H \times N \to N$ in which $\alpha\hspace{-0.8mm}\mid_{L(H) \times N}$ is an action of $L(H)$ on $N$, is compatible with the coarse quotient $Q$. If $\overline{L(H)}$ is not a two-sided ideal and $N \ne \{1\}$ then no function $\alpha \colon H \times N \to N$ is compatible with $Q$.
\end{proposition}

\begin{proof}
Suppose $\overline{L(H)}$ is an ideal and suppose $\alpha\hspace{-0.8mm}\mid_{L(H) \times N}$ is an action of $L(H)$ on $N$.

The first compatibility condition in \cref{actiondef} says that whenever $(n_1,h) \sim (n_2,h)$ that $(n_1\alpha(h,n),h) \sim (n_2\alpha(h,n),h)$ for all $n\in N$. If $h \in L(H)$ then $n_1 = n_2$ and so this holds. If $h \in \overline{L(H)}$ then all elements are related and so this also holds.

The second compatibility condition requires that if $(n_1,h') \sim (n_2,h')$ then $(\alpha(h,n_1),hh') \sim (\alpha(h,n_2),hh')$ for all $h \in H$. Now again if $h' \in L(H)$, $n_1 = n_2$ and we have that the condition is satisfied. If $h' \in \overline{L(H)}$ then $hh' \in \overline{L(H)}$ as $\overline{L(H)}$ is a two-sided ideal. Thus $(\alpha(h,n_1),hh') \sim (\alpha(h,n_2),hh')$, as all such elements are related.

We know that $\alpha\hspace{-0.8mm}\mid_{L(H) \times N}$ is an action of $L(H)$ on $N$ and so we have that for all $x\in L(H)$ and $n,n' \in N$:

\begin{enumerate}
\item $[\alpha(x,nn'),x] = [\alpha(x,n)\alpha(x,n'),x]$,
\item $[\alpha(1,n),1] = [n,1]$,
\item $[\alpha(x,1),x] = [1,x]$.	
\end{enumerate}

Since $Q$ is the coarse quotient we have that for all $y \in L(H)$ and $n,n' \in N$:

\begin{enumerate}
\item $[\alpha(y,nn'),y] = [\alpha(y,n)\alpha(y,n'),y]$,
\item $[\alpha(y,1),y] = [1,y]$.	
\end{enumerate}

Since $L(H)$ and $\overline{L(H)}$ are complements this means that the only condition remaining is to check that for all $h,h' \in H$ and $n \in N$ we have that $[\alpha(hh',n),hh'] = [\alpha(h,\alpha(h',n)),hh']$. Now if either $h \in \overline{L(H)}$ or $h' \in \overline{L(H)}$ we will have $hh' \in \overline{L(H)}$ which will immediately give equality. If neither $h$ nor $h'$ are elements of $\overline{L(H)}$ then they both belong to $L(H)$ and so using the fact that $\alpha\hspace{-0.8mm}\mid_{L(H) \times N}$ is an action of $L(H)$ on $N$ we have that $[\alpha(hh',n),hh'] = [\alpha(h,\alpha(h',n)),hh']$. Thus $\alpha$ is compatible with the coarse quotient $Q$.

If $N \ne \{1\}$ and $\overline{L(H)}$ is not an ideal then there exist elements $y \in \overline{L(H)}$ and $x \in L(H)$ such that $xy \in L(H)$. Now given some function $\alpha \colon H \times N \to N$, for the second compatibility condition to hold we need in particular that $(n_1,y) \sim (n_2,y)$ implies that $(\alpha(x,n_1),xy) \sim (\alpha(x,n_2),xy)$ which must then imply that $\alpha(x,n_1) = \alpha(x,n_2)$ for all $n_1, n_2 \in N$. But we also know that $\alpha(x,1) = 1$ and so $\alpha(x,n) = 1$ for all $n$. Now let $x^*$ be a left inverse of $x$. We know that for all $n \in N$ we have
\begin{align*}
[n,1] 		&= [\alpha(1,n),1] \\						&= [\alpha(xx^*,n),xx^*] \\ 				&= [\alpha(x,\alpha(x^*,n)),xx^*] \\ 
			&= [1,1].	
\end{align*}

This is a contradiction and so $\alpha$ cannot be compatible with $Q$.
\end{proof}

Whenever $H$ is finite, commutative, a group or has no nontrivial right invertible elements at all, then this right ideal $\overline{L(H)}$ will be a two-sided ideal. Similarly if $H$ is a monoid of $n \times n$ matrices over a field, then this same property holds. 

\begin{example}
Let $H$ and $N$ both be the monoid of $n \times n$ matrices with entries from some field $K$. Matrices with right inverses always have two-sided inverses and so have non-zero determinant. Thus the coarsest quotient $Q$ on $N \times H$ gives that when $\mathrm{det}(B) = 0$ we have that $(A,B) \sim (A',B)$ for all $A,A' \in N$ and when $\mathrm{det}(B) \ne 0$ we have $(A,B) \sim (A',B)$ if and only if $A = A'$. 

Observe that the submonoid $L(H)$ is actually the group $\mathrm{GL}(n,K)$. Thus we can consider the map $\alpha \colon H \times N \to N$ where $\alpha(B,A) = BAB^{-1}$ whenever $\mathrm{det}(B) \ne 0$ and $\alpha(B,A) = A$ otherwise. This is clearly an action of $\mathrm{GL}(n,K)$ on $N$.

Note that the right ideal $\overline{L(H)}$ is a two-sided ideal due to the multiplicative nature of the determinant. Thus this map $\alpha$ is compatible with the coarsest quotient $Q$.
\end{example}

This provides our first example of a weakly Schreier extension with a non-trivial quotient and non-trivial action.

\begin{example}
Suppose that besides the identity $H$ has no elements with right inverses. Then the disjoint union $N \sqcup (H - \{1\})$ can be equipped with the following multiplication and made a monoid. 
Let $n,n' \in N$ and $h,h' \in H - \{1\}$. Then

\begin{enumerate}
    \item $n \cdot n' = n \cdot_N n'$,
    \item $h \cdot h' = h \cdot_H h'$,
    \item $n \cdot h = h = h \cdot n$.
\end{enumerate}

The extension $\splitext{N}{k}{N \sqcup (H -\{1\})}{e}{s}{H}$ is weakly Schreier as it is isomorphic to the weakly Schreier extension given by $(Q,[\alpha])$ where $Q$ is the coarse quotient on $N \times H$ and $\alpha(h,n) = n$. The isomorphism is given by sending $n$ to $[n,1]$ and $h$ to $[1,h]$.
\end{example}

Crucial in the proof of these last few results is that we have partitioned the elements of $H$ into a submonoid and an ideal. Ideals whose complements are submonoids are called prime. We can generalise the above results for any prime ideal $Y$. Replace the coarse quotient with a new quotient $Q$ in which for $y \in Y$ we have that $(n,y) \sim (n',y)$ for all $n, n' \in N$ and for $x \in H-Y$ we have that $(n,x) \sim (n',x)$ implies that $n = n'$. Crucially a non-trivial prime ideal can never contain a right invertible element, for if it did it would contain the identity. Thus $Q$ is seen to be admissible and the above results can be seen to equally apply to this construction.

\subsection*{Acknowledgements}

I thank Graham Manuell for the many discussions had on this topic (all while he was writing up his PhD).

\bibliographystyle{abbrv}
\bibliography{bibliography}
\end{document}